\documentclass[a4paper,11pt]{amsart}

\usepackage[colorlinks,linkcolor=blue,citecolor=blue]{hyperref}
\usepackage{latexsym, amssymb, amsmath, amsthm, mathrsfs, bbm, pstricks, ifthen}
\usepackage{orcidlink}

\usepackage{tikz}
\usepackage{tikz-cd}

\usepackage[all, knot]{xy}
\xyoption{arc}
\allowdisplaybreaks[4]

\newcommand{\Dchaintwo}[4]{
\rule[-3\unitlength]{0pt}{8\unitlength}
\begin{picture}(14,5)(0,3)
\put(1,2){\ifthenelse{\equal{#1}{l}}{\circle*{2}}{\circle{2}}}
\put(2,2){\line(1,0){10}}
\put(13,2){\ifthenelse{\equal{#1}{r}}{\circle*{2}}{\circle{2}}}
\put(1,5){\makebox[0pt]{\scriptsize #2}}
\put(7,4){\makebox[0pt]{\scriptsize #3}}
\put(13,5){\makebox[0pt]{\scriptsize #4}}
\end{picture}}

\newcommand{\Hexagon}[7]{
\rule[-3\unitlength]{0pt}{8\unitlength}
\begin{picture}(30,20)(0,3)
\put(1,13){\ifthenelse{\equal{#1}{l}}{\circle*{2}}{\circle{2}}}
\put(2,14){\line(1,1){6}}
\put(9,21){\ifthenelse{\equal{#1}{l}}{\circle*{2}}{\circle{2}}}
\put(10,21){\line(1,0){8}}
\put(19,21){\ifthenelse{\equal{#1}{l}}{\circle*{2}}{\circle{2}}}
\put(20,20){\line(1,-1){6}}
\put(27,13){\ifthenelse{\equal{#1}{l}}{\circle*{2}}{\circle{2}}}
\put(26,12){\line(-1,-1){6}}
\put(19,5){\ifthenelse{\equal{#1}{l}}{\circle*{2}}{\circle{2}}}
\put(18,5){\line(-1,0){8}}
\put(9,5){\ifthenelse{\equal{#1}{l}}{\circle*{2}}{\circle{2}}}
\put(8,6){\line(-1,1){6}}
\put(-3,12){\makebox[0pt]{\scriptsize #2}}
\put(9,23){\makebox[0pt]{\scriptsize #3}}
\put(19,23){\makebox[0pt]{\scriptsize #4}}
\put(31,12){\makebox[0pt]{\scriptsize #5}}
\put(19,1){\makebox[0pt]{\scriptsize #6}}
\put(9,1){\makebox[0pt]{\scriptsize #7}}
\end{picture}}

\newcommand{\THexagon}[7]{
\rule[-3\unitlength]{0pt}{8\unitlength}
\begin{picture}(30,20)(0,3)
\put(1,13){\ifthenelse{\equal{#1}{l}}{\circle*{2}}{\circle{2}}}
\put(2,14){\line(1,1){6}}
\put(9,21){\ifthenelse{\equal{#1}{l}}{\circle*{2}}{\circle{2}}}
\put(10,21){\line(1,0){8}}
\put(19,21){\ifthenelse{\equal{#1}{l}}{\circle*{2}}{\circle{2}}}
\put(20,20){\line(1,-1){6}}
\put(27,13){\ifthenelse{\equal{#1}{l}}{\circle*{2}}{\circle{2}}}
\put(26,12){\line(-1,-1){6}}
\put(19,5){\ifthenelse{\equal{#1}{l}}{\circle*{2}}{\circle{2}}}
\put(18,5){\line(-1,0){8}}
\put(9,5){\ifthenelse{\equal{#1}{l}}{\circle*{2}}{\circle{2}}}
\put(8,6){\line(-1,1){6}}
\put(2,13){\line(1,0){24}}
\put(10,21){\line(1,-2){8}}
\put(18,21){\line(-1,-2){8}}
\put(-3,12){\makebox[0pt]{\scriptsize #2}}
\put(9,23){\makebox[0pt]{\scriptsize #3}}
\put(19,23){\makebox[0pt]{\scriptsize #4}}
\put(31,12){\makebox[0pt]{\scriptsize #5}}
\put(19,1){\makebox[0pt]{\scriptsize #6}}
\put(9,1){\makebox[0pt]{\scriptsize #7}}
\end{picture}}

\newcommand{\Dchainfour}[8]{
\rule[-3\unitlength]{0pt}{5\unitlength}
\begin{picture}(38,5)(0,3)
\put(1,2){\ifthenelse{\equal{#1}{1}}{\circle*{2}}{\circle{2}}}
\put(2,2){\line(1,0){10}}
\put(13,2){\ifthenelse{\equal{#1}{2}}{\circle*{2}}{\circle{2}}}
\put(14,2){\line(1,0){10}}
\put(25,2){\ifthenelse{\equal{#1}{3}}{\circle*{2}}{\circle{2}}}
\put(26,2){\line(1,0){10}}
\put(37,2){\ifthenelse{\equal{#1}{4}}{\circle*{2}}{\circle{2}}}
\put(1,5){\makebox[0pt]{\scriptsize #2}}
\put(7,4){\makebox[0pt]{\scriptsize #3}}
\put(13,5){\makebox[0pt]{\scriptsize #4}}
\put(19,4){\makebox[0pt]{\scriptsize #5}}
\put(25,5){\makebox[0pt]{\scriptsize #6}}
\put(31,4){\makebox[0pt]{\scriptsize #7}}
\put(37,5){\makebox[0pt]{\scriptsize #8}}
\end{picture}}

\def \1{\mathbf{1}}

\numberwithin{equation}{section}

\newtheorem{theorem}{Theorem}[section]
\newtheorem{lemma}[theorem]{Lemma}
\newtheorem{proposition}[theorem]{Proposition}

\newtheorem{definition}[theorem]{Definition}
\newtheorem{example}[theorem]{Example}

\begin{document}

\title{Quotient Category of a Multiring Category}
\thanks{This work was supported by National Natural Science Foundation of China(No.12271243).}

\author{Zhenbang Zuo~\orcidlink{0009-0008-9013-0365}, Gongxiang Liu}

\address{Department of Mathematics, Nanjing University, Nanjing 210093, China}
\email{mg20210045@smail.nju.edu.cn}

\address{Department of Mathematics, Nanjing University, Nanjing 210093, China}
\email{gxliu@nju.edu.cn}

\date{}
\maketitle

\begin{abstract}
The aim of this paper is to introduce a tensor structure for the Serre quotient category of an abelian monoidal category with biexact tensor product to make the canonical functor a monoidal functor. In this tensor product, the Serre quotient category of a multiring category (resp. a multitensor category) by a two-sided Serre tensor-ideal is still a multiring category (resp. a multitensor category). Besides, a two-sided Serre tensor-ideal of a tensor category is always trivial. This result can be generalized to any tensor product. If the canonical functor is a monoidal functor, then the corresponding Serre subcategory of the tensor category is trivial. 
\end{abstract}

\section{Introduction}
This paper builds upon the groundwork of \cite{gabriel1962categories}, in which Gabriel introduced the quotient category $\mathcal{A}/\mathcal{C}$ of an abelian category $\mathcal{A}$ by a Serre subcategory $\mathcal{C}$. This quotient category is often called a Serre quotient category, and the process of obtaining a Serre quotient category is often called the localization of an abelian category. Gabriel proved that the quotient category of an abelian category is still an abelian category, and the canonical functor $T:\mathcal{A} \to \mathcal{A}/\mathcal{C}$ is an exact functor. 

Serre quotient categories attract many mathematicians. We collect several recent results about Serre quotient category as follows. An analog of the the fundamental homomorphism theorem is introduced by Mohamed Barakat and Markus Lange-Hegermann in \cite{MR3200455}. They also show that the coimage of a Gabriel monad is a Serre quotient category in \cite{MR3138372}, and research the Ext-computability of Serre quotient categories in \cite{MR3261464}. Ramin Ebrahimi shows that the natural map $q_{X,A}^i:\mathrm{Ext}_{\mathcal{A}}^i(X,A) \to \mathrm{Ext}_{\mathcal{A}/\mathcal{C}}^i(q(X),q(A))$ is invertible in \cite{MR4512513}, where $q$ is the canonical functor and $X,A$ are objects in $\mathcal{A}$. Moreover, Xiao-Wu Chen and Henning Krause define the expansion of abelian categories in \cite{MR2811570}.

We aim to introduce a tensor product for the Serre quotient category of a multitensor category. To make the results seem more interesting, we divide our results into two parts. The first part focuses on cases of abelian category, in which we show that the quotient category of a locally finite abelian category is a locally finite abelian category. In fact, the first part provides conditions for the second part.

The second part focuses on cases of monoidal category, in which we provide a definition for tensor product in the quotient category of an abelian monoidal category with biexact tensor product, and show that the quotient category of a multiring category (resp. a multitensor category) by a two-sided Serre tensor-ideal is a multiring category (resp. a multitensor category). Moreover, we show that a two-sided Serre tensor-ideal of a multiring category with left duals must be a direct sum of the component subcategories, and therefore the corresponding quotient category is isomorphic to a direct sum of the component subcategories of the original category.

In this paper, we present some basic definitions and results as preparation in section 2. We discuss the locally finiteness of the quotient category of a locally finite abelian category in section 3 as the first part of our main results. We devote section 4 to give a definition for tensor product in the quotient category of an abelian monoidal category with biexact tensor product, to study the quotient category of a multiring category (resp. a multitensor category), and to study the structure of a two-sided Serre tensor-ideal of a multiring category with left duals as the second part of our main results.

\section{Preliminaries}
Let $k$ be an algebraically closed field throughout this paper. Recall the definition of Serre subcategory.

\begin{definition}[Serre subcategory]
Let $\mathcal{A}$ be an abelian category. A full additive subcategory $\mathcal{C} \subset \mathcal{A}$ is a Serre
subcategory provided that $\mathcal{C}$ is closed under taking subobjects, quotients and extensions.
\end{definition}

Note that there is an equivalent definition for Serre subcategory.

\begin{lemma}
Let $\mathcal{A}$ be an abelian category. A full subcategory $\mathcal{C}$ of $\mathcal{A}$ is a Serre subcategory if and only if for any $X'\to X \to X''$ exact in $\mathcal{A}$ with $X',X''\in \mathcal{C}$, then also $X \in \mathcal{C}$.
\end{lemma}

Recall the definition of quotient category of an abelian category. It is also called Serre quotient category. One can refer to \cite{gabriel1962categories} for more details about the localization of an abelian category.

\begin{definition}[Quotient Category]
Let $\mathcal{A}$ be an abelian category, $\mathcal{C}$ be a Serre subcategory of $\mathcal{A}$. The category $\mathcal{A}/\mathcal{C}$ is defined as follows:
\begin{enumerate}
    \item the objects of $\mathcal{A}/\mathcal{C}$ coincide with the objects of $\mathcal{A}$.
    \item the set of morphisms from $M$ to $N$ is defined by
\end{enumerate}
$$
\mathrm{Hom}_{\mathcal{A}/\mathcal{C}}(M,N) := \varinjlim_{M',N'} \mathrm{Hom}_{\mathcal{A}}(M',N/N')
$$
in which $M'$ and $N'$ go through the subobjects of $M$ and $N$ respectively such that $M/M', N' \in \mathcal{C}$.
\end{definition}

Denote the canonical functor by $T:\mathcal{A} \to \mathcal{A}/\mathcal{C}$. In fact, every morphism $\bar{f}$ in $\mathcal{A}/\mathcal{C}$ can be written as $Tf$ for some $f$ in the direct system. One can observe this fact by either understanding that the direct system in the definition of quotient category is directed or referring to \cite{gabriel1962categories}. 

The following lemma, which is Lemma 2 in \cite{gabriel1962categories}, describes how we can determine whether $Tu$ is a zero morphism, a monomorphism, or an epimorphism in the quotient category $\mathcal{A}/\mathcal{C}$ for a given morphism $u$ in $\mathcal{A}$.

\begin{lemma}
Let $\mathcal{A}$ be an abelian category, $\mathcal{C}$ be a Serre subcategory of $\mathcal{A}$. Let $u:M\to N$ be a morphism of $\mathcal{A}$, then
\begin{enumerate}
    \item $Tu$ is zero if and only if $\mathrm{Im}\ u \in \mathcal{C}$;
    \item $Tu$ is a monomorphism if and only if $\mathrm{Ker}\ u \in \mathcal{C}$;
    \item $Tu$ is an epimorphism if and only if $\mathrm{Coker}\ u \in \mathcal{C}$;
\end{enumerate}
\end{lemma}

Recall that an additive category $\mathcal{A}$ is said to be $k$-linear if for any objects $X,Y \in \mathcal{A}$, $\mathrm{Hom}_{\mathcal{A}}(X,Y)$ is equipped with a structure of a vector space over $k$, such that composition of morphisms is $k$-linear. The following definitions refer to \cite{MR3242743}.

\begin{definition}[Locally finite]
A $k$-linear abelian category $\mathcal{A}$ is said to be locally finite if the following two conditions are satisfied:
\begin{enumerate}
    \item for any two objects $X,Y$ in $\mathcal{A}$, the vector space $\mathrm{Hom}_{\mathcal{A}}(X,Y)$ is finite dimensional;
    \item every object in $\mathcal{A}$ has finite length.
\end{enumerate}
\end{definition}

Recall also the definitions for multiring category and multitensor category as follows.

\begin{definition}
\begin{enumerate}
\item A multiring category $\mathcal{A}$ over $k$ is a locally finite $k$-linear abelian monoidal category with bilinear and biexact tensor product. If in addition $\mathrm{End}_{\mathcal{A}}(1) = k$, we will call $\mathcal{A}$ a ring category.
 
\item A multitensor category $\mathcal{A}$ over $k$ is a locally finite $k$-linear abelian rigid monoidal category with bilinear tensor product. If in addition $\mathrm{End}_{\mathcal{A}}(1) = k$, and $\mathcal{A}$ is indecomposable i.e. $\mathcal{A}$ is not equivalent to a direct sum of nonzero multitensor categories, then we will call $\mathcal{A}$ a tensor category.
\end{enumerate}
\end{definition}

In fact, the tensor product of a multitensor category is always biexact, see Proposition 4.2.1 in \cite{MR3242743}. This implies a multitensor category is always a multiring category.

Recall also Theorem 4.3.8 in \cite{MR3242743} as the following:  

\begin{theorem}
\begin{enumerate}
\item In a ring category with left duals, the unit object $1$ is simple.
\item In a multiring category with left duals, the unit object $1$ is semisimple, and is a direct sum of pairwise non-isomorphic simple objects $1_i$.
\end{enumerate}
\end{theorem}

Note that $1_i$ is the image of $p_i$, where $\{p_i \}_{i \in I}$ is the set of primitive idempotents of the algebra $\mathrm{End}(1)$. And $1= \bigoplus\limits_{i\in I} 1_i$.

We also write down the definition of two-sided Serre tensor-ideal. In some literature, for example in \cite{MR4136434}, tensor-ideal is written as $\otimes$-ideal, and it means one side absorption. In some other literature, for example in \cite{MR4477957}, the name tensor-ideal is used. 

\begin{definition}
Let $\mathcal{A}$ be an abelian monoidal category. A Serre subcategory $\mathcal{C}$ of $\mathcal{A}$ is called a two-sided Serre tensor-ideal of $\mathcal{A}$ if for any $X \in \mathcal{A}$, $Y \in \mathcal{C}$, we have $X\otimes Y \in \mathcal{C}$ and $Y\otimes X \in \mathcal{C}$.
\end{definition}

\section{Localization preserves locally finiteness}

Let $\mathcal{A}$ be an abelian category with a Serre subcategory $\mathcal{C}$. As proved in \cite{gabriel1962categories}, $\mathcal{A}/\mathcal{C}$ is an abelian category. This conclusion allows us to discuss the length of an object in $\mathcal{A}/\mathcal{C}$.

We have the following lemma about simple objects in $\mathcal{A}/\mathcal{C}$.

\begin{lemma}
Every simple object in $\mathcal{A}$ has length $1$ or $0$ in $\mathcal{A}/\mathcal{C}$.    
\end{lemma}

\begin{proof}
Suppose $X$ is a simple object in $\mathcal{A}$, consider its subobject $Y$ in $\mathcal{A}/\mathcal{C}$, and denote the corresponding monomorphism by $\bar{f}:Y \to X$.

We also view $Y$ as an object in $\mathcal{A}$. Because 
$$
\mathrm{Hom}_{\mathcal{A}/\mathcal{C}}(Y,X) = \varinjlim_{Y',X'} \mathrm{Hom}_{\mathcal{A}}(Y',X/X'),
$$
where $Y/Y',\ X' \in \mathcal{C}$, there exists a subobject $Y_1$ of $Y$ such that $Y/Y_1 \in \mathcal{C}$ such that the following diagram commute

$$
\xymatrix{
TY_1 \ar[r]^{Ti} \ar[d]_{Tf} & TY \ar[d]^{\bar{f}}\\
TX \ar@{=}[r] & TX
}
$$
where $T$ is the canonical functor, $i$ is the injection, and $f:Y_1 \to X$ is an epimorphism since $X$ is simple in $\mathcal{A}$. Thus, we have the following exact sequence
$$
\xymatrix{
0 \ar[r] & \mathrm{Ker}\ f \ar[r] & Y_1 \ar[r]^{f} & X \ar[r] & 0. 
}
$$
Because the canonical functor $T$ is exact, we have the following exact sequence
$$
\xymatrix{
0 \ar[r] & T\mathrm{Ker}\ f \ar[r] & TY_1 \ar[r]^{Tf} & TX \ar[r] & 0. 
}
$$
This means $Tf$ is an epimorphism in $\mathcal{A}/\mathcal{C}$. Note that $Ti$ is an isomorphism in $\mathcal{A}/\mathcal{C}$ and $\bar{f}$ is a monomorphism, we know $Tf= \bar{f}\circ Ti$ is a monomorphism. Therefore, $Tf$ is an isomorphism. Consequently, $\bar{f}$ is an isomorphism. Thus, every monomorphism in $\mathcal{A}/\mathcal{C}$ to $X$ is an isomorphism in $\mathcal{A}/\mathcal{C}$. It shows that $X$ has length $1$ or $0$ in $\mathcal{A}/\mathcal{C}$. Note that the case of length $0$ means $X$ is in $\mathcal{C}$, and it is zero in $\mathcal{A}/\mathcal{C}$.
\end{proof}

Next, we prove that the localization preserves finite length.

\begin{proposition}
Suppose every object in $\mathcal{A}$ has finite length, then every object in $\mathcal{A}/\mathcal{C}$ has finite length.
\end{proposition}

\begin{proof}
For an arbitrary object $X$ in $\mathcal{A}$, it has finite length in $\mathcal{A}$ and one can assume its Jordan-Hölder series as following without loss of generality.
$$
0=X_0 \subset X_1 \subset \cdots \subset X_{n-1} \subset X_n = X
$$
Since $X_{i+1}/X_i$ is a simple object in $\mathcal{A}$, it has length $1$ or $0$ in $\mathcal{A}/\mathcal{C}$ by the above Lemma. Because 
$$
    l(X_{i+1})=l(X_i)+l(X_{i+1}/X_i),\  \text{ for all }\  0 \leq i \leq n-1,
$$
we know that
$$
\begin{aligned}
    l(X) &=l(X_n) =l(X_{n-1})+l(X_n/X_{n-1})\\ 
    &= l(X_{n-2})+l(X_{n-1}/T_{n-2})+l(X_{n-1}/T_n)\\
    &= \cdots\\
    &= \sum_{i=0}^{n-1} l(X_{i+1}/X_i)\\
    &\leq n.
\end{aligned}
$$
This means $X$ has finite length in $\mathcal{A}/\mathcal{C}$, which is actually smaller than its length in $\mathcal{A}$.
\end{proof}

In order to show that the quotient category of a locally finite abelian category is locally finite, we are required to prove that the Hom-spaces of the quotient category are finite dimensional.

\begin{lemma}
Let $\mathcal{A}$ be a locally finite $k$-linear abelian category, then $\mathrm{Hom}_{\mathcal{A}/\mathcal{C}}(M,N)$ is a finite dimensional vector space, where $M,N$ are objects in $\mathcal{A}$.
\end{lemma}

\begin{proof}
By definition, $\mathrm{Hom}_{\mathcal{A}/\mathcal{C}}(M,N)$ is obviously a vector space. We now prove $\mathrm{Hom}_{\mathcal{A}/\mathcal{C}}(M,N)$ is finite dimensional by induction on the lengths of $M,N$ in $\mathcal{A}$. Firstly, suppose $l(M)=l(N)=1$. If $M \in \mathcal{C}$ or $N \in \mathcal{C}$, then $\mathrm{Hom}_{\mathcal{A}/\mathcal{C}}(M,N)=0$. If $M,N \notin \mathcal{C}$, it is clear that $\mathrm{Hom}_{\mathcal{A}/\mathcal{C}}(M,N) = \mathrm{Hom}_{\mathcal{A}}(M,N)$, which is finite dimensional.

Secondly, we want to show $\mathrm{Hom}_{\mathcal{A}/\mathcal{C}}(M,N)$ is a finite dimensional vector space for $l(N)=1$. We prove it by induction on the length of $M$, suppose $\mathrm{Hom}_{\mathcal{A}/\mathcal{C}}(M,N)$ is finite dimensional for all $M,\ N$ such that $l(M)\leq m$ and $l(N)=1$. Now we consider the case of $l(M)= m+1$ and $l(N)=1$. On one hand, if there is a subobject $X$ of $M$ such that $0 \leq l(X) \leq m$ and $M/X \in \mathcal{C}$, we know that $Ti_M^X$ is an isomorphism in $\mathcal{A}/\mathcal{C}$, where $i_M^X:X \to M$ is the monomorphism. This means that $\mathrm{Hom}_{\mathcal{A}/\mathcal{C}}(M,N) \cong \mathrm{Hom}_{\mathcal{A}/\mathcal{C}}(X,N)$, which is finite dimensional by the induction assumption. On the other hand, if for any subobject $X$ of $M$ satisfying $0 \leq l(X) \leq m$, $M/X \notin \mathcal{C}$. It follows that $\mathrm{Hom}_{\mathcal{A}/\mathcal{C}}(M,N) = \mathrm{Hom}_{\mathcal{A}}(M,N)$, which is finite dimensional. Therefore, $\mathrm{Hom}_{\mathcal{A}/\mathcal{C}}(M,N)$ is a finite dimensional vector space for $l(N)=1$. Similarly, $\mathrm{Hom}_{\mathcal{A}/\mathcal{C}}(M,N)$ is a finite dimensional vector space for $l(M)=1$.

Thirdly, given two positive integers $m,n \geq 1$, suppose $\mathrm{Hom}_{\mathcal{A}/\mathcal{C}}(M,N)$ is a finite dimensional vector space for $l(M)\leq m$ and $l(N)\leq n+1$ and for $l(M)\leq m+1$ and $l(N)\leq n$. We consider the case of $l(M)= m+1$ and $l(N)= n+1$. 

If there is a subobject $X$ of $M$ such that $0 \leq l(X) \leq m$ and $M/X \in \mathcal{C}$, then $\mathrm{Hom}_{\mathcal{A}/\mathcal{C}}(M,N) \cong \mathrm{Hom}_{\mathcal{A}/\mathcal{C}}(X,N)$ since $Ti_M^X$ is an isomorphism in $\mathcal{A}/\mathcal{C}$, where $i_M^X:X \to M$ is the monomorphism. By the induction assumption, $\mathrm{Hom}_{\mathcal{A}/\mathcal{C}}(M,N) \cong \mathrm{Hom}_{\mathcal{A}/\mathcal{C}}(X,N)$ is finite dimensional.

If there is a subobject $Y$ of $N$ such that $1 \leq l(Y) \leq n+1$ and $Y \in \mathcal{C}$, then $\mathrm{Hom}_{\mathcal{A}/\mathcal{C}}(M,N) \cong \mathrm{Hom}_{\mathcal{A}/\mathcal{C}}(M,N/Y)$ since $Tp_{N/Y}^N$ is an isomorphism in $\mathcal{A}/\mathcal{C}$, where $p_{N/Y}^N:N \to N/Y$ is the epimorphism. By the induction assumption, $\mathrm{Hom}_{\mathcal{A}/\mathcal{C}}(M,N) \cong \mathrm{Hom}_{\mathcal{A}/\mathcal{C}}(M,N/Y)$ is finite dimensional.

If for any subobject $X$ of $M$ and any subobject $Y$ of $N$ satisfying $0 \leq l(X) \leq m$ and $1 \leq l(Y) \leq n+1$, $M/X, Y \notin \mathcal{C}$. It follows that $\mathrm{Hom}_{\mathcal{A}/\mathcal{C}}(M,N) = \mathrm{Hom}_{\mathcal{A}}(M,N)$, which is finite dimensional. 

In summary, $\mathrm{Hom}_{\mathcal{A}/\mathcal{C}}(M,N)$ is a finite dimensional vector space for $l(M)= m+1$ and $l(N)= n+1$. Consequently, we obtain that $\mathrm{Hom}_{\mathcal{A}/\mathcal{C}}(M,N)$ is a finite dimensional vector space for any objects $M,N$ in $\mathcal{A}$.
\end{proof}

It follows directly from Proposition 3.2 and Lemma 3.3 that $\mathcal{A}/\mathcal{C}$ is locally finite. 

\begin{proposition}
The quotient category $\mathcal{A}/\mathcal{C}$ of a locally finite $k$-linear abelian category $\mathcal{A}$ is a locally finite $k$-linear abelian category. 
\end{proposition}

\section{Tensor product in Serre quotient category}

\subsection{Definition}
In order to define the tensor product in the Serre quotient category of an abelian monoidal category, we first introduce a lemma.

\begin{lemma}
Let $\mathcal{A}$ be an abelian monoidal category, $\mathcal{C}$ be a two-sided Serre tensor-ideal of $\mathcal{A}$. Suppose $f,g$ are morphisms in $\mathcal{A}$ such that $\mathrm{ker}\ f$, $\mathrm{ker}\ g$ belong to $\mathcal{C}$, then $\mathrm{ker}(f\otimes g)$ belongs to $\mathcal{C}$. Similarly, if $\mathrm{coker}\ f$, $\mathrm{coker}\ g$ belong to $\mathcal{C}$, then $\mathrm{coker}(f\otimes g)$ belongs to $\mathcal{C}$.
\end{lemma}

\begin{proof}
Consider
$$
f:M \to N \ \text{and}\ g:X\to Y.
$$
Note that there is a composition
$$
\begin{tikzcd}
f\otimes g: M\otimes X \arrow[r,"f\otimes id_X"] & N\otimes X \arrow[r,"id_N\otimes g"] & N\otimes Y,
\end{tikzcd}
$$
and there is a commutative diagram
$$
\begin{tikzcd}
& M\otimes X \arrow[r,"f\otimes id"] \arrow[d,"f\otimes g"] & N\otimes X \arrow[r] \arrow[d,"id\otimes g"] & \mathrm{coker}(f\otimes id) \arrow[r] \arrow[d] & 0\\
0 \arrow[r] & N\otimes Y \arrow[r,"id"] & N\otimes Y \arrow[r] & 0 &
\end{tikzcd}
$$
By snake lemma, this implies an exact sequence
$$
\xymatrix{
0 \ar[r] & \mathrm{ker}(f\otimes id) \ar[r] & \mathrm{ker}(f\otimes g) \ar[r] & \mathrm{ker}(id\otimes g) \ar[llld] \\
 \mathrm{coker}(f\otimes id)  \ar[r] & \mathrm{coker}(f\otimes g) \ar[r] & \mathrm{coker}(id\otimes g) \ar[r] & 0
}
$$
which means 
$$
\left\{
\begin{aligned}
&\mathrm{ker}(f\otimes id) \to  \mathrm{ker}(f\otimes g) \to  \mathrm{ker}(id\otimes g)\\
&\mathrm{coker}(f\otimes id) \to \mathrm{coker}(f\otimes g) \to  \mathrm{coker}(id\otimes g)
\end{aligned}
\right.
$$
exact. Note that $\mathrm{ker}(f\otimes id)=\mathrm{ker}\ f\otimes X \in \mathcal{C}$, and $\mathrm{ker}(id\otimes g)=N\otimes \mathrm{ker}\ g \in \mathcal{C}$ provided $\mathrm{ker}\ f$, $\mathrm{ker}\ g$ belong to $\mathcal{C}$. By Lemma 2.2, we know $\mathrm{ker}(f\otimes g) \in \mathcal{C}$. Similarly, $\mathrm{coker}(f\otimes g) \in \mathcal{C}$ provided $\mathrm{coker}\ f$, $\mathrm{coker}\ g$ belong to $\mathcal{C}$.
\end{proof}

The following two lemmas study the tensor product of two subobjects, and the tensor product of two quotient objects respectively.

\begin{lemma}
Let $\mathcal{A}$ be an abelian monoidal category with biexact tensor product. Let $M$, $X$ be two objects in $\mathcal{A}$, $M'$ be a subobject of $M$, $X'$ be a subobject of $X$, then $M'\otimes X'$ is a subobject of $M\otimes X$.
\end{lemma}

\begin{proof}
Consider monomorphisms $i_1:M' \to M$ and $i_2:X'\to X$, and exact sequences
$$
\xymatrix{
0 \ar[r] & M' \ar[r]^{i_1} & M \ar[r] & \mathrm{coker}\ i_1 \ar[r] & 0; \\
0 \ar[r] & X' \ar[r]^{i_2} & X \ar[r] & \mathrm{coker}\ i_2 \ar[r] & 0.
}
$$
Because the tensor product is biexact, we have two following exact sequences
$$
\xymatrix{
0 \ar[r] & M' \otimes X' \ar[r]^{i_1 \otimes id} & M \otimes X' \ar[r] & \mathrm{coker}\ i_1 \otimes X' \ar[r] & 0; \\
0 \ar[r] & M\otimes X' \ar[r]^{id \otimes i_2} & M\otimes X \ar[r] & M\otimes \mathrm{coker}\ i_2 \ar[r] & 0.
}
$$
This means
$$
\begin{tikzcd}
i_1 \otimes i_2: M' \otimes X' \arrow[r,"i_1 \otimes id"] & M \otimes X' \arrow[r,"id \otimes i_2"] & M \otimes X
\end{tikzcd}
$$
is a monomorphism. Hence, $M'\otimes X'$ is a subobject of $M\otimes X$.
\end{proof}

\begin{lemma}
Let $\mathcal{A}$ be an abelian monoidal category with biexact tensor product. Let $N$, $Y$ be two objects in $\mathcal{A}$, $N/N'$ be a quotient object of $N$, $Y/Y'$ be a quotient object of $Y$, then $N/N'\otimes Y/Y'$ is a quotient object of $N\otimes Y$.
\end{lemma}

\begin{proof}
Consider epimorphisms $p_1:N \to N/N'$ and $p_2:Y\to Y/Y'$, and exact sequences
$$
\xymatrix{
0 \ar[r] & N' \ar[r] & N \ar[r]^{p_1} & N/N' \ar[r] & 0; \\
0 \ar[r] & Y' \ar[r] & Y \ar[r]^{p_2} & Y/Y' \ar[r] & 0; \\
}
$$
Because the tensor product is biexact, we have two following exact sequences
$$
\xymatrix{
0 \ar[r] & N' \otimes Y/Y' \ar[r] & N \otimes Y/Y' \ar[r]^{p_1 \otimes id} & N/N' \otimes Y/Y' \ar[r] & 0; \\
0 \ar[r] & N \otimes Y' \ar[r] & N\otimes Y \ar[r]^{id \otimes p_2} & N\otimes Y/Y' \ar[r] & 0.
}
$$
This means
$$
\begin{tikzcd}
p_1 \otimes p_2: N\otimes Y \arrow[r,"id \otimes p_2"] & N\otimes Y/Y' \arrow[r,"p_1 \otimes id"] & N/N' \otimes Y/Y'
\end{tikzcd}
$$
is an epimorphism. Hence, $N/N'\otimes Y/Y'$ is a quotient object of $N\otimes Y$.
\end{proof}

In fact, we have an isomorphism
$$
N/N' \otimes Y/Y' \cong N\otimes Y/\mathrm{ker}(p_1 \otimes p_2).
$$

Now, we can prove that $T(f\otimes g) \in \mathrm{Hom}_{\mathcal{A}/\mathcal{C}}(M\otimes X, N\otimes Y)$ for any $\bar{f}=Tf:M\to N$, $\bar{g}=Tg:X \to Y$ in $\mathcal{A}/\mathcal{C}$.

\begin{proposition}
Let $\mathcal{A}$ be an abelian monoidal category with biexact tensor product, $\mathcal{C}$ be a two-sided Serre tensor-ideal of $\mathcal{A}$. Let $\bar{f}:M\to N$, $\bar{g}:X \to Y$ be two morphisms in $\mathcal{A}/\mathcal{C}$. Then $T(f\otimes g)$ is a morphism in $\mathrm{Hom}_{\mathcal{A}/\mathcal{C}}(M\otimes X, N\otimes Y)$ where $f,g$ are in direct systems, and $T$ is the canonical functor.
\end{proposition}

\begin{proof}
Suppose $f:M'\to N/N'$, $g:X' \to Y/Y'$ with $M/M'$, $N'$, $X/X'$, $Y'$ are in $\mathcal{C}$. It follows that
$$
f \otimes g: M'\otimes X' \to N/N'\otimes Y/Y'.
$$
Denote $i_1:M'\to M$, $i_2:X'\to X$ to be inclusions. Since $Ti_1$, $Ti_2$ are isomorphisms in $\mathcal{A}/\mathcal{C}$, we know that $\mathrm{coker}\ i_1$, $\mathrm{coker}\ i_2$ belong to $\mathcal{C}$. By Lemma 4.1, $\mathrm{coker}(i_1\otimes i_2)\in \mathcal{C}$. This means $(M\otimes X)/ (M'\otimes X') \in \mathcal{C}$.

Denote $p_1:N\to N/N'$, $p_2:Y\to Y/Y'$. We have already known that $N/N' \otimes Y/Y' \cong N\otimes Y/\mathrm{ker}(p_1 \otimes p_2)$. Since $Tp_1$, $Tp_2$ are isomorphisms in $\mathcal{A}/\mathcal{C}$, we know that $\mathrm{ker}\ p_1$, $\mathrm{ker}\ p_2$ are  in $\mathcal{C}$. By Lemma 4.1, it follows that $\mathrm{ker}(p_1\otimes p_2) \in \mathcal{C}$.

In summary, $f \otimes g$ is in the direct system. Therefore, $T(f\otimes g) \in \mathrm{Hom}_{\mathcal{A}/\mathcal{C}}(M\otimes X, N\otimes Y)$.
\end{proof}

\begin{proposition}
Let $\mathcal{A}$ be an abelian monoidal category with biexact tensor product, $\mathcal{C}$ be a two-sided Serre tensor-ideal of $\mathcal{A}$. Let $\bar{f}:M\to N$, $\bar{g}:X \to Y$ be two morphisms in $\mathcal{A}/\mathcal{C}$. Suppose $\bar{f}=Tf=Tf_1$ and $\bar{g}=Tg=Tg_1$, then $T(f\otimes g)= T(f_1\otimes g_1)$.
\end{proposition}

\begin{proof}
First of all, we claim that
$$
T(id\otimes g) =  T(id\otimes g_1). 
$$
Suppose $g:X' \to Y/Y'$, $g_1:X_1' \to Y/Y_1'$. Since the direct system in the definition of quotient categroty is directed, we can obtain a morphism $g_2:X_2'\to Y/Y_2'$ such that the following diagrams commute:
$$
\begin{tikzcd}
X_2' \arrow[r,"i^{X_2'}_{X'}"] \arrow[d,"g_2"] & X' \arrow[d,"g"] \\
Y/Y_2' & Y/Y' \arrow[l,"p^{Y/Y'}_{Y/Y_2'}"] \\
\end{tikzcd}
$$
and
$$
\begin{tikzcd}
X_2' \arrow[r,"i^{X_2'}_{X_1'}"] \arrow[d,"g_2"] & X_1' \arrow[d,"g_1"] \\
Y/Y_2' & Y/Y_1' \arrow[l,"p^{Y/Y_1'}_{Y/Y_2'}"]. \\
\end{tikzcd}
$$
Thus, we obtain two commutative diagrams:
$$
\begin{tikzcd}
M\otimes X_2' \arrow[r,"id\otimes i^{X_2'}_{X'}"] \arrow[d,"id\otimes g_2"] & M\otimes X' \arrow[d,"id\otimes g"] \\
M\otimes Y/Y_2' & M\otimes Y/Y' \arrow[l,"id\otimes p^{Y/Y'}_{Y/Y_2'}"] \\
\end{tikzcd}
$$
and
$$
\begin{tikzcd}
M\otimes X_2' \arrow[r,"id\otimes i^{X_2'}_{X_1'}"] \arrow[d,"id\otimes g_2"] & M\otimes X_1' \arrow[d,"id\otimes g_1"] \\
M\otimes Y/Y_2' & M\otimes Y/Y_1' \arrow[l,"id\otimes p^{Y/Y_1'}_{Y/Y_2'}"]. \\
\end{tikzcd}
$$
Because $id\otimes i^{X_2'}_{X'}= i^{M\otimes X_2'}_{M\otimes X'}$, $id\otimes i^{X_2'}_{X_1'}= i^{M\otimes X_2'}_{M\otimes X_1'}$, $id\otimes p^{Y/Y'}_{Y/Y_2'} =p^{M \otimes Y/Y'}_{M \otimes Y/Y_2'}$ and $id\otimes p^{Y/Y_1'}_{Y/Y_2'} =p^{M \otimes Y/Y_1'}_{M \otimes Y/Y_2'}$, we obtain two commutative diagrams:
$$
\begin{tikzcd}
M\otimes X_2' \arrow[r,"i^{M\otimes X_2'}_{M\otimes X'}"] \arrow[d,"id\otimes g_2"] & M\otimes X' \arrow[d,"id\otimes g"] \\
M\otimes Y/Y_2' & M\otimes Y/Y' \arrow[l,"p^{M \otimes Y/Y'}_{M \otimes Y/Y_2'}"] \\
\end{tikzcd}
$$
and
$$
\begin{tikzcd}
M\otimes X_2' \arrow[r,"i^{M\otimes X_2'}_{M\otimes X_1'}"] \arrow[d,"id\otimes g_2"] & M\otimes X_1' \arrow[d,"id\otimes g_1"] \\
M\otimes Y/Y_2' & M\otimes Y/Y_1' \arrow[l,"p^{M \otimes Y/Y_1'}_{M \otimes Y/Y_2'}"]. \\
\end{tikzcd}
$$
This means 
$$
T(id\otimes g)= T(id\otimes g_2) = T(id\otimes g_1).
$$
The claim has been proven, and similarly one can obtain that
$$
T(f\otimes id) =  T(f_1\otimes id). 
$$
Therefore,
$$
\begin{aligned}
T(f\otimes g) &= T((f\otimes id)\circ (id \otimes g))= T(f\otimes id)\circ T(id \otimes g)\\
&= T(f_1 \otimes id)\circ T(id \otimes g_1) = T((f_1 \otimes id)\circ (id \otimes g_1))\\ 
&= T(f_1\otimes g_1).
\end{aligned}
$$
\end{proof}

Hence, we can define the tensor product of two morphisms in the quotient category. The above propositions guarantee that the following definition is well-defined.

\begin{definition}
Let $\mathcal{A}$ be an abelian monoidal category with biexact tensor product, $\mathcal{C}$ be a two-sided Serre tensor-ideal of $\mathcal{A}$. Define the tensor product of objects in $\mathcal{A}/\mathcal{C}$ by the same tensor product in $\mathcal{A}$, and define the tensor product of morphisms in $\mathcal{A}/\mathcal{C}$ by
$$
\bar{f} \otimes \bar{g}:=T(f\otimes g)
$$
where $f,g$ are in direct systems, and $T$ is the canonical functor. It is clear from the definition that 
$$
Tf \otimes Tg = T(f\otimes g).
$$
Additionally, we define the associativity constraint in $\mathcal{A}/\mathcal{C}$ by
$$
\bar{a}_{X,Y,Z}=Ta_{X,Y,Z}.
$$
By the way, it is also clear that the left and right unit isomorphisms in $\mathcal{A}/\mathcal{C}$ are $Tl_X$ and $Tr_X$ where $l_X$ and $r_X$ are the left and right unit isomorphisms in $\mathcal{A}$.
\end{definition}

Let's consider an example of two-sided Serre tensor-ideal.
\begin{example}
Let $I$ be a finite indexed set, $\mathcal{A}= \bigoplus\limits_{i\in I} \mathrm{Mat}_{n_i}(\mathrm{Vec})$, where $\mathrm{Mat}_{n_i}(\mathrm{Vec})$ is the category whose objects are $n_i$-by-$n_i$ matrices of finite dimensional vector spaces. The tensor product of two objects from distinct direct summands is defined to be zero. One can observe that $\mathcal{A}$ is a multitensor category, and each $\mathrm{Mat}_{n_i}(\mathrm{Vec})$ is a two-sided Serre tensor-ideal of $\mathcal{A}$. In fact, let $\mathcal{C} = \mathrm{Mat}_{n_i}(\mathrm{Vec})$ to be such a two-sided Serre tensor-ideal, then $\mathcal{A}/\mathcal{C} = \bigoplus\limits_{j\in I \backslash \{i \} } \mathrm{Mat}_{n_i}(\mathrm{Vec})$ 
\end{example}

Next, we show that, with the above tensor product, the quotient category of a monoidal category is also a monoidal category.

\begin{proposition}
Let $\mathcal{A}$ be an abelian monoidal category with biexact tensor product, $\mathcal{C}$ be a two-sided Serre tensor-ideal of $\mathcal{A}$, then $\mathcal{A}/\mathcal{C}$ is a monoidal category. 
\end{proposition}

\begin{proof}
Consider the following pentagon axiom diagram of $\mathcal{A}$:
$$
\xymatrix{
& ((W\otimes X)\otimes Y)\otimes Z \ar[ld]_{a_{W,X,Y}\otimes id_Z} \ar[rd]^{a_{W\otimes X,Y,Z}} & \\
(W\otimes (X\otimes Y))\otimes Z \ar[d]^{a_{W,X\otimes Y,Z}} & & (W\otimes X)\otimes (Y\otimes Z) \ar[d]_{a_{W,X,Y\otimes Z}} \\
W\otimes((X\otimes Y)\otimes Z) \ar[rr]^{id_W\otimes a_{X,Y,Z}} & & W\otimes (X\otimes (Y\otimes Z))
}
$$
Applying the canonical functor $T$ gives that
$$
\xymatrix{
& ((W\otimes X)\otimes Y)\otimes Z \ar[ld]_{\bar{a}_{W,X,Y}\otimes id_Z} \ar[rd]^{\bar{a}_{W\otimes X,Y,Z}} & \\
(W\otimes (X\otimes Y))\otimes Z \ar[d]^{\bar{a}_{W,X\otimes Y,Z}} & & (W\otimes X)\otimes (Y\otimes Z) \ar[d]_{\bar{a}_{W,X,Y\otimes Z}} \\
W\otimes((X\otimes Y)\otimes Z) \ar[rr]^{id_W\otimes \bar{a}_{X,Y,Z}} & & W\otimes (X\otimes (Y\otimes Z)).
}
$$
This is actually the pentagon axiom of $\mathcal{A}/\mathcal{C}$.

Consider the following triangle axiom diagram of $\mathcal{A}$:
$$
\xymatrix{
(X\otimes 1) \otimes Y \ar[rr]^{a_{X,1,Y}} \ar[dr]_{r_X\otimes id_Y} & & X\otimes (1\otimes Y) \ar[dl]^{id_X\otimes l_Y} \\
& X\otimes Y & 
}
$$
Applying the canonical functor $T$ gives that
$$
\xymatrix{
(X\otimes 1) \otimes Y \ar[rr]^{\bar{a}_{X,1,Y}} \ar[dr]_{Tr_X\otimes id_Y} & & X\otimes (1\otimes Y) \ar[dl]^{id_X\otimes Tl_Y} \\
& X\otimes Y & .
}
$$
Thus, $\mathcal{A}/\mathcal{C}$ is a monoidal category.
\end{proof}

Recall that $Tf \otimes Tg = T(f\otimes g)$, this implies that $T$ is a monoidal functor.

\subsection{Localization of a multiring category}

Now, we study the quotient category of a multiring category (resp. a multitensor category) by a two-sided Serre tensor-ideal. 

\begin{proposition}
Let $\mathcal{A}$ be a multiring category, $\mathcal{C}$ be a two-sided Serre tensor-ideal of $\mathcal{A}$. Then $\mathcal{A}/\mathcal{C}$ is a multiring category. 
\end{proposition}

\begin{proof}
By Proposition 3.4 and Proposition 4.8, the quotient category of a locally finite $k$-linear abelian monoidal category is a locally finite $k$-linear abelian monoidal category. Therefore, it suffices to show that the tensor product $\otimes: \mathcal{A}/\mathcal{C} \times \mathcal{A}/\mathcal{C} \to \mathcal{A}/\mathcal{C}$ is bilinear and biexact.

Firstly, we show that the tensor product is bilinear. Note that the canonical functor $T$ is linear, thus 
$$
\begin{aligned}
&(\bar{f_1}+\bar{f_2}) \otimes \bar{g} 
= (Tf_1 +Tf_2) \otimes Tg 
= T(f_1+f_2)\otimes Tg\\ 
= &T((f_1+f_2)\otimes g)
= T(f_1\otimes g+f_2\otimes g)
= T(f_1\otimes g)+T(f_2\otimes g)\\
= &T(f_1)\otimes Tg+T(f_2)\otimes Tg
= \bar{f_1}\otimes \bar{g}+\bar{f_2}\otimes \bar{g}\\
\end{aligned}
$$
and similarly
$$
\begin{aligned}
&\bar{f} \otimes (\bar{g_1}+\bar{g_2}) 
= Tf \otimes (Tg_1+Tg_2)
= Tf\otimes T(g_1+g_2)\\
= &T(f\otimes (g_1+g_2))
= T(f\otimes g_1+ f\otimes g_2)
= T(f\otimes g_1)+ T(f\otimes g_2)\\
= &Tf\otimes Tg_1+Tf\otimes Tg_2
= \bar{f}\otimes \bar{g_1} + \bar{f}\otimes \bar{g_2}.
\end{aligned}
$$
Furthermore, for any $a \in k$, we have
$$
\begin{aligned}
&a\bar{f}\otimes \bar{g} 
= aTf \otimes Tg
= T(af)\otimes Tg
= T(af\otimes g)\\
= &T(f\otimes ag)
= Tf\otimes T(ag)
= Tf\otimes aTg
= \bar{f}\otimes a\bar{g}.
\end{aligned}
$$
This implies $\otimes: \mathcal{A}/\mathcal{C} \times \mathcal{A}/\mathcal{C} \to \mathcal{A}/\mathcal{C}$ is bilinear on morphisms.

Secondly, we show that the tensor product is biexact. Consider the following exact sequence
$$
\xymatrix{
L \ar[r]^{\bar{f}} & M \ar[r]^{\bar{g}} & N .
}
$$
Because we can write $\bar{g}=Tg$, $\bar{f}=Tf$, the exact sequence means that
$$
\mathrm{ker}(Tg)=\mathrm{Im}(Tf).
$$
By Lemma 3 in \cite{gabriel1962categories}, it follows that
$$
T(\mathrm{ker}\ g)= T(\mathrm{Im} \ f).
$$
Because $T$ is an exact functor, the short exact sequence
$$
\begin{tikzcd}
0 \arrow[r] & \mathrm{Im} f \arrow[r, "i"] & \mathrm{ker}\ g \arrow[r, "\pi"] & \mathrm{ker}\ g/ \mathrm{Im} \ f \arrow[r] & 0
\end{tikzcd}
$$
implies that
$$
\begin{tikzcd}
0 \arrow[r] & T\mathrm{Im} f \arrow[r, "Ti"] & T\mathrm{ker}\ g \arrow[r, "T\pi"] & T(\mathrm{ker}\ g/ \mathrm{Im} \ f) \arrow[r] & 0
\end{tikzcd}
$$
is exact. Consequently,
$$
T(\mathrm{ker}(g)/\mathrm{Im}(f)) \cong T (\mathrm{ker}\ g)/T(\mathrm{Im}\ f) =0 \  \text{in} \ \mathcal{A}/\mathcal{C},
$$
which means $\mathrm{ker}(g)/\mathrm{Im}(f) \in \mathcal{C}$.

For any object $A$ in $\mathcal{A}$, since $A\otimes -$ is exact in $\mathcal{A}$, we have the following short exact sequence
$$
\begin{tikzcd}
0 \arrow[r] & A\otimes \mathrm{Im} f \arrow[r,"id_A\otimes i"] & A\otimes \mathrm{ker}\ g \arrow[r,"id_A\otimes \pi"] & A\otimes (\mathrm{ker}\ g / \mathrm{Im} f) \arrow[r] & 0.
\end{tikzcd}
$$
Note that there is a short exact sequence
$$
\begin{tikzcd}
0 \arrow[r] & A\otimes \mathrm{Im} f \arrow[r,"id_A\otimes i"] & A\otimes \mathrm{ker}\ g \arrow[r,"\pi'"] & (A\otimes \mathrm{ker}\ g)/(A\otimes \mathrm{Im} f) \arrow[r] & 0.
\end{tikzcd}
$$
Therefore,
$$
(A\otimes \mathrm{ker}\ g)/(A\otimes \mathrm{Im} f) \cong A\otimes (\mathrm{ker}\ g / \mathrm{Im} f) \in \mathcal{C}.
$$

Now, we claim that
$$
A\otimes \mathrm{ker}\ g = \mathrm{ker} (id_A \otimes \ g).
$$
Consider the following exact sequence
$$
\xymatrix{
0 \ar[r] & \mathrm{ker}\ g \ar[r] & M \ar[r]^g & N,
}
$$
applying $A\otimes -$ gives the following exact sequence
$$
\xymatrix{
0 \ar[r] & A\otimes \mathrm{ker}\ g \ar[r] & A\otimes M \ar[r]^{id_A\otimes g} & A\otimes N.
}
$$
This indicates that 
$$
A\otimes \mathrm{ker}\ g = \mathrm{ker} (id_A \otimes g).
$$

Besides, note that 
$$
A\otimes \mathrm{Im} f = \mathrm{Im} (id_A \otimes f).
$$
It follows that
$$
\mathrm{ker} (id_A \otimes g)/\mathrm{Im} (id_A \otimes f) = (A\otimes \mathrm{ker}\ g)/(A\otimes \mathrm{Im} f) \in \mathcal{C}.
$$
This implies that
$$
T(\mathrm{ker} (id_A \otimes g)/\mathrm{Im} (id_A \otimes f))=0
$$
i.e.
$$
T(\mathrm{ker} (id_A \otimes g))/ T(\mathrm{Im} (id_A \otimes f))=0
$$
i.e.
$$
T(\mathrm{ker} (id_A \otimes g))= T(\mathrm{Im} (id_A \otimes f))
$$
By Lemma 3 in \cite{gabriel1962categories}, this means
$$
\mathrm{ker} (T(id_A \otimes g))= \mathrm{Im} (T(id_A \otimes f))
$$
i.e.
$$
\mathrm{ker} (id_A \otimes Tg)= \mathrm{Im} (id_A \otimes Tf)
$$
i.e.
$$
\mathrm{ker} (id_A \otimes \bar{g})= \mathrm{Im} (id_A \otimes \bar{f}).
$$
This means $A\otimes -$ is exact in $\mathcal{A}/\mathcal{C}$. Similarly, one can show $- \otimes A$ is exact in $\mathcal{A}/\mathcal{C}$. As a result, the tensor product is biexact in $\mathcal{A}/\mathcal{C}$. Thus, $\mathcal{A}/\mathcal{C}$ is a multiring category.
\end{proof}

\begin{proposition}
Let $\mathcal{A}$ be a multitensor category, $\mathcal{C}$ be a two-sided Serre tensor-ideal of $\mathcal{A}$, then $\mathcal{A}/\mathcal{C}$ is a multitensor category. 
\end{proposition}

\begin{proof}
As proved in the above proposition, we know that $\mathcal{A}/\mathcal{C}$ is a locally finite $k$-linear abelian monoidal category with bilinear tensor product. Therefore, it suffices to show $\mathcal{A}/\mathcal{C}$ is rigid.

For any object $X$ in $\mathcal{A}/\mathcal{C}$, it has a left dual $X^*$, which means there exist an evaluation $ev_X: X^*\otimes X \to 1$ and a coevaluation $coev_X:1 \to X \otimes X^*$ such that the compositions
$$
\begin{tikzcd}
        X \arrow[rr,"coev_X\otimes id_X"] & & (X\otimes X^*) \otimes X  \arrow[r,"a_{X,X^*,X}"] & X\otimes (X^* \otimes X) \arrow[rr,"id_X \otimes ev_X"] & & X,
\end{tikzcd}
$$
$$
\begin{tikzcd}
        X^* \arrow[rr,"id_{X^*} \otimes coev_X"] & & X^* \otimes (X\otimes X^*)   \arrow[r,"a_{X^*,X,X^*}^{-1}"] & (X^* \otimes X) \otimes X^* \arrow[rr,"ev_X\otimes id_{X^*}"] & & X^*
\end{tikzcd}
$$
are the identity morphisms. Applying $T$ gives that the compositions
$$
\begin{tikzcd}
        X \arrow[rr,"Tcoev_X\otimes id_X"] & & (X\otimes X^*) \otimes X  \arrow[r,"\bar{a}_{X,X^*,X}"] & X\otimes (X^* \otimes X) \arrow[rr,"id_X \otimes Tev_X"] & & X,
\end{tikzcd}
$$
$$
\begin{tikzcd}
        X^* \arrow[rr,"id_{X^*} \otimes Tcoev_X"] & & X^* \otimes (X\otimes X^*)   \arrow[r,"\bar{a}_{X^*,X,X^*}^{-1}"] & (X^* \otimes X) \otimes X^* \arrow[rr,"Tev_X\otimes id_{X^*}"] & & X^*
\end{tikzcd}
$$
are the identity morphisms. As a result, every object in $\mathcal{A}/\mathcal{C}$ has a left dual. Similarly, one can show every object in $\mathcal{A}/\mathcal{C}$ has a right dual. Thus, $\mathcal{A}/\mathcal{C}$ is rigid. Consequently, $\mathcal{A}/\mathcal{C}$ is a multitensor category.
\end{proof}

\subsection{Two-sided Serre tensor-ideal}

The following proposition shows that a two-sided Serre tensor-ideal of a tensor category is always trivial.

\begin{proposition}
Let $\mathcal{A}$ be a tensor category, $\mathcal{C}$ be a two-sided Serre tensor-ideal of $\mathcal{A}$, then $\mathcal{C}$ is trivial.
\end{proposition}

\begin{proof}
Suppose $\mathcal{C}$ is not zero, choose a non-zero object $B$ in $\mathcal{C}$. Because $\mathcal{C}$ is a two-sided Serre tensor-ideal, $B^* \otimes B \in \mathcal{C}$. 
Note that
$$
ev_B:B^* \otimes B \to 1
$$
is a non-zero epimorphism in $\mathcal{A}$, since $1$ is simple and the composition $(id_B\otimes ev_B) \circ a_{B,B^*,B}\circ (coev_B\otimes id_B)$ is identity. Therefore, $1 \in \mathcal{C}$. Consequently, $\mathcal{A} = \mathcal{C}$.
\end{proof}

This proposition means that $\mathcal{A}/\mathcal{C}$ could only be zero or $\mathcal{A}$ itself. 

However, someone may be interested in other definitions of tensor structures of $\mathcal{A}/\mathcal{C}$ to make $\mathcal{A}/\mathcal{C}$ be a monoidal category. Interestingly, we will prove next that no matter how we define the tensor structure, the canonical functor being monoidal implies the Serre subcategory is trivial.

\begin{proposition}
Let $\mathcal{A}$ be a tensor category, $\mathcal{C}$ be a Serre subcategory of $\mathcal{A}$. Suppose the canonical functor $T$ is a monoidal functor, then $\mathcal{C}$ is trivial.
\end{proposition}

\begin{proof}
We prove it by contradiction. Assume $\mathcal{C}$ is non-trivial.

Choose a non-zero object $B$ in $\mathcal{C}$, we know that $TB=0$ in $\mathcal{A}/\mathcal{C}$. This means $T(B^* \otimes B)\cong TB^* \otimes TB = 0$ in $\mathcal{A}/\mathcal{C}$. However, 
$$
ev_B:B^* \otimes B \to 1
$$
is a non-zero epimorphism in $\mathcal{A}$ because $1$ is a simple object in $\mathcal{A}$. Therefore, 
$$
T(ev_B): T(B^* \otimes B) \cong TB^* \otimes TB = 0 \to T1
$$
is an epimorphism by Lemma 2.4. Since $T$ is a monoidal functor, $T1\not= 0$ in $\mathcal{A}/\mathcal{C}$. However, this contradicts $T(ev_B)$ is an epimorphism in $\mathcal{A}/\mathcal{C}$.
\end{proof}

Now, we consider the representation category of $k\mathbb{Z}_2$.

\begin{example}
Consider $\mathbb{Z}_2 = \{1,g \}$ and a field $k$ such that $\mathrm{char}\ k \nmid 2$, and denote the representation category of $k\mathbb{Z}_2$ by $\mathcal{A}$. As is well known, $\mathcal{A}$ is semisimple and it has only two irreducible representations $W_1=k(1-g)$ and $W_2=k(1+g)$. Define a homomorphism by
$$
\begin{aligned}
\varphi: &k(1-g)\otimes k(1-g) &\to &k(1+g) \\
&a(1-g)\otimes (1-g) &\mapsto &a(1+g)
\end{aligned}
$$
where $a \in k$.

Let $\mathcal{C}$ be a Serre subcategory containing $W_1$. We now show that if $\mathcal{A}/\mathcal{C}$ is a tensor category, then $\mathcal{C}=\mathcal{A}$.  Since $W_1 \in \mathcal{C}$, we know that $TW_1=0$ in $\mathcal{A}/\mathcal{C}$. It follows that $TW_1\otimes TW_1\cong T(W_1\otimes W_1) =T(k(1-g)\otimes k(1-g)) = 0$ in $\mathcal{A}/\mathcal{C}$. This implies $W_1\otimes W_1 \in \mathcal{C}$. Consequently, $W_2=k(1+g) \in \mathcal{C}$ because $\varphi$ is a $k\mathbb{Z}_2$-isomorphism. Thus, both $W_1$ and $W_2$ are in $\mathcal{C}$, and $\mathcal{C}=\mathcal{A}$.

In order to make $\mathcal{A}/\mathcal{C}$ be a tensor category, it is clear that $W_2 = k(1+g)$ cannot be in $\mathcal{C}$, because $W_2 = k(1+g)$ is the unit object in $\mathcal{A}$. 
\end{example}

Next, we show that a two-sided Serre tensor-ideal of a multiring category is a direct sum of some component subcategories. First of all, recall that a multiring category $\mathcal{A}$ can be written as a direct sum of its component subcategories $\mathcal{A}=\bigoplus\limits_{i,j \in I} \mathcal{A}_{i,j}$, where $\mathcal{A}_{i,j}= 1_i \otimes \mathcal{A} \otimes 1_j$ and $I$ is an indexed set such that $1 = \bigoplus\limits_{i\in I} 1_i$.

\begin{lemma}
Let $\mathcal{A}$ be a multiring category with left duals, $\mathcal{C}$ be a two-sided Serre tensor-ideal of $\mathcal{A}$, then $\mathcal{C} \cap \mathcal{A}_{i,j}$ is either $\mathcal{A}_{i,j}$ or $0$, where $\mathcal{A}_{i,j}$ is a component subcategory of $\mathcal{A}$.
\end{lemma}

\begin{proof}
Suppose $\mathcal{C} \cap \mathcal{A}_{i,j} \not= 0$. Choose $X \not= 0$ in $\mathcal{C} \cap \mathcal{A}_{i,j}$, we know that $X = 1_i \otimes X \otimes 1_j$. Because $X^*= (1_i \otimes X \otimes 1_j)^* = 1_j^* \otimes X^* \otimes 1_i^* =1_j \otimes X^* \otimes 1_i$, one can observe that
$$
(1_i \otimes X \otimes 1_j)^* \otimes 1_i \otimes X \otimes 1_j \in \mathcal{C} \cap \mathcal{A}_{j,j}.
$$

We claim that $\mathrm{Im} (ev_X) = 1_j$. Assume $1_k$ is a direct summand of $\mathrm{Im} (ev_X)$, where $k \not= j$. Consider the exact sequence
$$
\begin{tikzcd}
X^* \otimes X \arrow[r,"ev_X"] & \mathrm{Im} (ev_X) \arrow[r] & 0.
\end{tikzcd}
$$
Tensoring this sequence with $1_k$ on the left, we obtain an exact sequence
$$
\begin{tikzcd}
1_k \otimes X^* \otimes X \arrow[r,"id_{1_k} \otimes ev_X"] & 1_k \arrow[r] & 0.
\end{tikzcd}
$$
Since
$$
1_k \otimes X^* \otimes X = 1_k \otimes 1_j \otimes X^* \otimes 1_i \otimes 1_i \otimes X \otimes 1_j =0,
$$
the above exact sequence means $1_k = 0$, which is absurd. Therefore, $1_k$ is not a direct summand of $\mathrm{Im} (ev_X)$. Because $ev_X: X^* \otimes X \to 1$ is non-zero, we get that $\mathrm{Im} (ev_X) = 1_j$. Consequently, $1_j \in \mathcal{C}$. This implies $\mathcal{A}_{i,j} \subset \mathcal{C}$, and thus $\mathcal{C} \cap \mathcal{A}_{i,j} = \mathcal{A}_{i,j}$.
\end{proof}

In fact, if $\mathcal{A}_{i,j} \subset \mathcal{C}$, one can know $1_j\in \mathcal{C}$ from the above proof. As a result, $\mathcal{A}_{l,j} \subset \mathcal{C}$ and $\mathcal{A}_{j,l} \subset \mathcal{C}$ for all $l \in I$. Similar to the process of the above proof, if we consider
$$
coev_X:1 \to X \otimes X^* = 1_i \otimes X \otimes 1_j \otimes 1_j \otimes X^* \otimes 1_i,
$$
then we can obtain $1_i \in \mathcal{C}$. Consequently, $\mathcal{A}_{l,i} \subset \mathcal{C}$ and $\mathcal{A}_{i,l} \subset \mathcal{C}$ for all $l \in I$.

In particular, if $0 \not= \mathcal{A}_{i,j} \subset \mathcal{C}$, then $0 \not= \mathcal{A}_{i,i} \subset \mathcal{C}$, $0 \not= \mathcal{A}_{j,j} \subset \mathcal{C}$, and $0 \not= \mathcal{A}_{j,i} \subset \mathcal{C}$. Besides,
$$
\begin{aligned}
\mathcal{C} &= \mathcal{C} \cap \mathcal{A}\\
&= \mathcal{C} \cap \bigoplus_{i,j \in I} \mathcal{A}_{i,j}\\
&= \bigoplus_{i,j \in I} (\mathcal{C} \cap \mathcal{A}_{i,j}).
\end{aligned}
$$
It follows from the above Lemma that $\mathcal{C}$ is a direct sum of $\mathcal{A}_{i,j}$'s.

The following proposition provides a deeper understanding for two-sided Serre tensor-ideal.

\begin{proposition}
Let $\mathcal{A}$ be a multiring category with left duals, $\mathcal{C}$ be a two-sided Serre tensor-ideal of $\mathcal{A}$. Let $J=\{ i\in I| 1_i \in \mathcal{C} \}$, then $\mathcal{A}_{i,k}=0$, $\mathcal{A}_{k,i}=0$ for all $k \notin J$, $i \in J$.
\end{proposition}

\begin{proof}
Assume $\mathcal{A}_{i,k} \not= 0$ for a given $i \in J$ and $k \notin J$. Since $1_i \in \mathcal{C}$, we know that $\mathcal{A}_{i,k} \subset \mathcal{C}$. For any $0 \not= X \in \mathcal{A}_{i,k}$, we can obtain $\mathrm{Im} (ev_X)=1_k \in \mathcal{C}$ from a process similar to the proof of Lemma 4.14. This contradicts $k \notin J$. Thus, $\mathcal{A}_{i,k} = 0$. Similarly, $\mathcal{A}_{k,i} = 0$. 
\end{proof}

Let $\mathcal{A}$ be a multiring category with left duals, the above proposition implies that a two-sided Serre tensor-ideal $\mathcal{C}$ of $\mathcal{A}$ can be written as $\mathcal{C} = \bigoplus\limits_{i,j\in J} \mathcal{A}_{i,j}$, where $J=\{ i\in I| 1_i \in \mathcal{C} \}$. Furthermore, for any $0 \not= X \in \mathcal{A}$, the above proposition indicates that
$$
\begin{aligned}
X &= \bigoplus_{i,j \in I} (1_i \otimes X \otimes 1_j)\\
&= \bigoplus_{i,j \notin J} (1_i \otimes X \otimes 1_j) \bigoplus \bigoplus_{i,j \in J} (1_i \otimes X \otimes 1_j).
\end{aligned}
$$

Let
$$
X' = \bigoplus_{i,j \notin J} (1_i \otimes X \otimes 1_j) \text{ and } X''=\bigoplus_{i,j \in J} (1_i \otimes X \otimes 1_j),
$$
it is clear that $X'' \in \mathcal{C}$ and $X' \cong X$ in $\mathcal{A}/\mathcal{C}$. Therefore, $\mathcal{A}/\mathcal{C} \cong \bigoplus\limits_{i,j\notin J} \mathcal{A}_{i,j}$.

Conversely, given a suitable subset $J$ of $I$, is $\bigoplus\limits_{i,j\in J} \mathcal{A}_{i,j}$ a two-sided Serre tensor-ideal? We show next that the answer is yes.

\begin{proposition}
Let $\mathcal{A}$ be a multiring category with left duals, $J \subset I$ satisfying that $\mathcal{A}_{i,k}=0$, $\mathcal{A}_{k,i}=0$ for all $k \notin J$, $i \in J$. Then $\mathcal{C}= \bigoplus\limits_{i,j\in J} \mathcal{A}_{i,j}$ is a two-sided Serre tensor-ideal of $\mathcal{A}$.
\end{proposition}

\begin{proof}
It suffices to show $\mathcal{C}$ is a Serre subcategory. For any $X \in \mathcal{C}$, let $Y$ be a subobject of $X$. We know that 
$$
\begin{tikzcd}
0 \arrow[r] & Y \arrow[r] & X
\end{tikzcd}
$$
is exact. For $p \notin J$ or $q \notin J$, tensoring $1_p$ and $1_q$ on the left and right respectively gives that
$$
\begin{tikzcd}
0 \arrow[r] & 1_p \otimes Y \otimes 1_q \arrow[r] & 1_p \otimes X \otimes 1_q
\end{tikzcd}
$$
is exact. Because $1_p \otimes X \otimes 1_q$ is $0$, we obtain $1_p \otimes Y \otimes 1_q = 0$. This means $Y \in \mathcal{C}$. Hence, $\mathcal{C}$ is closed under taking subobjects. Similarly, $\mathcal{C}$ is closed under taking quotient objects. Now, suppose there is an exact sequence
$$
\begin{tikzcd}
0 \arrow[r] & X \arrow[r] & Y \arrow[r] & Z \arrow[r] & 0
\end{tikzcd}
$$
where $X,Z \in \mathcal{C}$. For $p \notin J$ or $q \notin J$, tensoring $1_p$ and $1_q$ on the left and right respectively gives that
$$
\begin{tikzcd}
0 \arrow[r] & 1_p \otimes X \otimes 1_q \arrow[r] & 1_p \otimes Y \otimes 1_q \arrow[r] & 1_p \otimes Z \otimes 1_q \arrow[r] & 0
\end{tikzcd}
$$
is exact. Because $1_p \otimes X \otimes 1_q = 0 = 1_p \otimes Z \otimes 1_q$, we obtain that $1_p \otimes Y \otimes 1_q = 0$. This means $Y \in \mathcal{C}$. Hence, $\mathcal{C}$ is closed under taking extensions. It follows that $\mathcal{C}$ is a Serre subcategory.
\end{proof}

In summary, on one hand for $J \subset I$ satisfying that $\mathcal{A}_{i,k}=0$, $\mathcal{A}_{k,i}=0$ for all $k \notin J$, $i \in J$, $\mathcal{C}= \bigoplus\limits_{i,j\in J} \mathcal{A}_{i,j}$ is a two-sided Serre tensor-ideal of $\mathcal{A}$. On the other hand, every two-sided Serre tensor-ideal of $\mathcal{A}$ can be written as $\mathcal{C} = \bigoplus\limits_{i,j\in J} \mathcal{A}_{i,j}$ for some $J \subset I$.

Consider the two-sided Serre tensor-ideal $\mathcal{C} = \bigoplus\limits_{i,j\in J} \mathcal{A}_{i,j}$. One can observe that the restriction of the canonical functor $T$ on $\bigoplus\limits_{i,j\notin J} \mathcal{A}_{i,j}$ is both an isomorphism and a monoidal functor. This implies the corresponding quotient category  $\mathcal{A}/\mathcal{C}$ is actually isomorphic to $\bigoplus\limits_{i,j\notin J} \mathcal{A}_{i,j}$ which is a subcategory of $\mathcal{A}$. Furthermore, it is easy to see that $T(1_{\mathcal{A}}) = T(\bigoplus_i 1_i) =  \bigoplus\limits_{i\notin J} 1_i$. Since $T$ is a monoidal functor, $1_{\mathcal{A}/\mathcal{C}} = T(1_{\mathcal{A}}) = \bigoplus\limits_{i\notin J} 1_i$. In fact, for $i,j\notin J$, $\mathcal{A}_{i,j}$ is a component subcategory of $\mathcal{A}/\mathcal{C}$.  

A direct corollary of the above result is that the image of another two-sided Serre tensor-ideal $\mathcal{C}'$ is a two-sided Serre tensor-ideal of the quotient category. Let $\mathcal{C} = \bigoplus\limits_{i,j\in J} \mathcal{A}_{i,j}$, $\mathcal{C}' = \bigoplus\limits_{i,j\in J'} \mathcal{A}_{i,j}$ be two two-sided Serre tensor-ideals of $\mathcal{A}$. Because $\mathcal{A}/\mathcal{C} \cong \bigoplus\limits_{i,j\notin J} \mathcal{A}_{i,j}$, $T(\mathcal{C}') \cong \bigoplus\limits_{i,j\in J' \backslash J} \mathcal{A}_{i,j}$ in $\mathcal{A}/\mathcal{C}$. We know that $\mathcal{A}_{i,k}=0$, $\mathcal{A}_{k,i}=0$ for all $k \notin J'$, $i \in J'$. Hence $\mathcal{A}_{i,k}=0$, $\mathcal{A}_{k,i}=0$ for all $k \notin J'\backslash J$, $i \in J'\backslash J$. This means the image of $\mathcal{C}'$ is a two-sided Serre tensor-ideal of $\mathcal{A}/\mathcal{C}$.

We end this section by discussing two-sided Serre tensor-ideals from a groupoid, one can refer the following example to section 4.13 in \cite{MR3242743}.

\begin{example}
Let $\mathcal{G}=(X,G,\mu,s,t,u,i)$ be a groupoid whose set of objects $X$ is finite and let $\mathcal{C}(\mathcal{G})$ be the category of finite dimensional vector spaces graded by the set $G$ of morphisms of $\mathcal{G}$ i.e. vector spaces of the form $V=\bigoplus\limits_{g\in G}V_g$. Introduce a tensor product on $\mathcal{C}(\mathcal{G})$ by the formula
$$
(V\otimes W)_g = \bigoplus_{(g_1,g_2):g_1g_2=g}V_{g_1}\otimes W_{g_2},
$$
where $V, W$ are objects in $\mathcal{C}(\mathcal{G})$.

We know that $\mathcal{C}(\mathcal{G})$ is a multitensor category. Suppose $A \in X$ is an object in $\mathcal{G}$ such that $\mathrm{Hom}_{\mathcal{G}}(A,B)$ is empty for all $B\not= A$. For convenience, we denote $G(A)=\mathrm{Hom}_{\mathcal{G}}(A,A)$. Now, we show that $\{ V=\bigoplus\limits_{g\in G(A)}V_g \}$ is a two-sided Serre tensor-ideal of $\mathcal{C}(\mathcal{G})$. For any object $W$ in $\mathcal{C}(\mathcal{G})$,
$$
((\bigoplus_{g\in G(A)}V_g) \otimes W)_k = \bigoplus_{(g_1,g_2):g_1g_2=k}(\bigoplus_{g\in G(A)}V_g)_{g_1}\otimes W_{g_2}.
$$
Note that $(\bigoplus\limits_{g\in G(A)}V_g)_{g_1}\otimes W_{g_2}$ is not zero only if $g_1 \in G(A)$. Consequently, it is not zero only if $k\in G(A)$. This means that $(\bigoplus\limits_{g\in G(A)}V_g) \otimes W \in \{ V=\bigoplus\limits_{g\in G(A)}V_g \}$. Similarly, $W \otimes (\bigoplus\limits_{g\in G(A)}V_g) \in \{ V=\bigoplus\limits_{g\in G(A)}V_g \}$. Thus, $\{ V=\bigoplus\limits_{g\in G(A)}V_g \}$ is a two-sided Serre tensor-ideal of $\mathcal{C}(\mathcal{G})$. 
\end{example}

In fact, $\{ V=\bigoplus\limits_{g\in G'}V_g \}$ is also a two-sided Serre tensor-ideal of $\mathcal{C}(\mathcal{G})$ if $G'$ is the morphism set of a connected component of $\mathcal{G}$. The above example is the case of a connected component consisting of one object.

\end{document}